\documentclass[12pt]{amsart}
\usepackage{mathrsfs}
\usepackage{amsthm,amsmath}
\usepackage{xypic}
\usepackage{rotating}
\usepackage{setspace}
\usepackage{pstricks}
\usepackage{graphicx}

\xyoption{all}

\usepackage{textcomp}
\usepackage{amsfonts}
\usepackage{amssymb}
\usepackage{amsthm}
\usepackage{stmaryrd}

\newcommand{\nin}{\not \in}

\newcommand{\reals}{\mathbb{R}}
\newcommand{\naturals}{\mathbb{N}}

\newcommand{\integers}{\mathbb{Z}}

\newcommand{\inv}[1]{{#1}^{-1}}

\newcommand{\kerm}{{\textrm{ ker }}}
\newcommand{\imm}{{\textrm{ im }}}

\newcommand{\autm}{\textrm{Aut}}

\newcommand{\iso}{\cong}
\newcommand{\boundaryinf}{\partial}

\newtheorem*{corollary*}{Corollary}
\newtheorem*{theorem*}{Theorem}
\newtheorem{theorem}{Theorem}[section]
\newtheorem{lemma}[theorem]{Lemma}

\newtheorem{proposition}[theorem]{Proposition}
\newtheorem{corollary}[theorem]{Corollary}
\newtheorem{observation}[theorem]{Observation}

\theoremstyle{definition}
\newtheorem{definition}{Definition}[section]
\newtheorem{example}{Example}[section]
\newtheorem*{example*}{Example}

\theoremstyle{remark}
\newtheorem*{remark}{Remark}

\title{Connectivity properties for actions on locally finite trees}
\author{
Keith Jones \\ 
Trinity College \\
{\tt keith.jones@trincoll.edu}
}

\begin{document}

\begin{abstract}
Given an action $G \stackrel{\rho}{\curvearrowright} T$ by a finitely
generated group on a locally finite tree, we view points of the visual boundary
$\boundaryinf T$ as directions in $T$ and use $\rho$ to lift
this sense of direction to $G$. For each
point $E \in \boundaryinf T$, this allows us to ask if $G$ is
$(n-1)$-connected ``in the direction of $E$''.  The
invariant $\Sigma^n(\rho)\subseteq \boundaryinf T$ then records the set of
directions in which $G$ is $(n-1)$-connected. In this paper, we introduce a
family of actions for which $\Sigma^1(\rho)$ can be calculated through 
analysis of certain quotient maps between trees. We show that for actions of
this sort, under reasonable hypotheses, $\Sigma^1(\rho)$ consists of
no more than a single point. By strengthening the hypotheses, we
are able to characterize precisely when a given end point lies in
$\Sigma^n(\rho)$ for any $n$. 
\end{abstract}
\maketitle

\section{Introduction}

Let $G$ be a group having type $F_n$,\footnote{By definition, $G$ has type $F_n$
iff there exists a $K(G,1)$-complex having finite $n$-skeleton. This
is equivalent to saying that there is an $n$-dimensional
$(n-1)$-connected CW-complex on which $G$ acts freely and
cocompactly by permuting cells. All groups have type $F_0$, while type $F_1$ is
equivalent to finitely generated and type $F_2$ is equivalent to
finitely presented \cite[\S7.2]{geogheganbook}.}
 and let $M$ be a proper CAT(0)
metric space.\footnote{A CAT(0) space is a geodesic metric space whose geodesic
triangles are no fatter than the corresponding ``comparison
triangles'' in the Euclidean plane, and a metric space is proper if every closed 
ball is compact \cite[Ch. II.1]{bridsonhaefliger}.} Let $\rho: G \rightarrow Isom(M)$ be an action by
isometries.
In \cite{amsmemoir}, Bieri and Geoghegan introduced a collection of geometric
``$\Sigma$-invariants'', $\Sigma^n(\rho)$, $n \geq 0$.  These arise naturally from
the study of the Bieri-Neumann-Strebel-Renz (BNSR) invariants $\Sigma^n(G)$,
which can then be viewed as a special case.  These invariants provide topological
insight into $\rho$ and provide algebraic
information about $G$.  In particular, if $\rho$ has
discrete orbits and $G$ is
finitely generated, then $\Sigma^1(\rho) = \boundaryinf M$ iff the point stabilizers
under $\rho$ are finitely generated; more generally, if $G$ has type
$F_n$, then $\Sigma^n(\rho) = \boundaryinf M$ iff the point stabilizers under $\rho$
have type $F_n$.\footnote{See Theorem A and the
Boundary Criterion in \cite{amsmemoir}, and note that the required
condition ``almost geodesically complete'' is ensured by cocompactness
due to Ontaneda \cite[Theorem B]{ontanedacocompact}.}
 
The invariant $\Sigma^n(\rho)$ depends on a notion of ``controlled
connectivity'', which we will briefly describe here.\footnote{The technical
 definition of controlled connectivity is provided in
\S\ref{controlledconnectivitysection}.}
The action $\rho$ can be used to impose a sense of direction
on $G$ as follows.  The space $M$ has a CAT(0) boundary $\boundaryinf
M$, which is in one to one correspondence with the collection of
geodesic rays emanating from any particular point of $M$.  In this
way, $\boundaryinf M$ encompasses the set of directions in $M$ in which one
can ``go to infinity''.  
For an end point $E \in \boundaryinf M$ there is a nested sequence of
subsets of $M$ (called horoballs about $E$). This nested sequence
provides a filtration of $M$.  Because $G$ has type $F_n$, there is an
$n$-dimensional $(n-1)$-connected CW-complex $X$ on which $G$ acts
freely and cocompactly by permuting cells.  One can then choose a
$G$-equivariant ``control'' map $h: X \rightarrow M$. Fixing an $E \in \boundaryinf M$,
$h$ allows us to lift the sense of direction from $M$ up to $X$ (and therefore $G$ by
proxy) by taking the preimages of horoballs about $E$.  If, roughly speaking,
 the preimages of the horoballs about $E$ are $(n-1)$-connected, 
the action $\rho$ is said to be {\em
controlled $(n-1)$-connected} or $CC^{n-1}$ over $E$.\footnote{For $n =0$, we
 take $(-1)$-connected to mean non-empty.} 
The precise definition
ensures that this is independent of choice of $X$ or $h$, and is in
fact a property of $\rho$  \cite[\S 3.2]{amsmemoir}.

For $n \geq 0$, the invariant $\Sigma^n(\rho)$ consists of
all those end points over which $\rho$ is $CC^{n-1}$. These form a
nested family  
\[
\Sigma^0(\rho) \supseteq \Sigma^1(\rho) \supseteq \Sigma^2(\rho) \dots
. \]
The action $\rho$ induces a topological action by $G$ on $\boundaryinf M$, under
which $\Sigma^n(\rho)$ is invariant. Those familiar with the BNSR
invariant $\Sigma^n(G)$ may recall that the BNSR invariant is an open
subset of the boundary, which in their case is a sphere. It is worth pointing out 
that the Bieri-Geoghegan
invariant $\Sigma^n(\rho)$ is in general not open in $\boundaryinf M$.

Bieri and Geoghegan have calculated $\Sigma^n$ for the modular group acting on
the hyperbolic plane in \cite{bierigeoghegansl2} and provide information
about $\Sigma^n$ for actions on trees by metabelian groups of finite
Pr\"ufer rank in \cite{amsmemoir}, Example C in Chapter 10. In his PhD thesis \cite{rehn}, 
Rehn provides 
calculations for the natural action by $SL_n(\integers[\frac{1}{k}])$ on the symmetric
space for $SL_n(\reals)$.  

In the case where $M = T$, a locally finite simplicial tree, calculations in 
\cite{amsmemoir} led Bieri and Geoghegan to ask whether
$\Sigma^1(\rho)$ would always be either empty, a singleton, or the entire
boundary of the tree. The ``entire boundary'' case has been discussed
above. In his Frankfurt Diploma Thesis \cite{lehnert}, Lehnert gave an 
example for which this is not the case. However, in this paper we illustrate 
that there does exist a class of actions for which
$\Sigma^n$ is either empty or a singleton. 

This paper is a development of part of the author's Ph.D dissertation at SUNY
Binghamton. The author is grateful to his Ph.D advisor Ross Geoghegan for his constant support and
encouragement.  Additionally, the clarity and elegance of this paper
have benefitted significantly from the suggestions of anonymous referees,
whose time and effort are greatly appreciated.

\subsection{Main Result}
All trees are assumed to be simplicial trees viewed as CAT(0) metric spaces
by giving each edge a length of 1. All actions under consideration are
by simplicial automorphisms and therefore are by isometries. Furthermore, we assume
that actions are without
inversions --- i.e, an edge is stabilized if and only if it is fixed
pointwise --- since we can simply pass to the barycentric subdivision
otherwise. Any tree exhibiting such an action by a group $G$ will be
referred to as a $G$-tree. All $G$-trees are assumed to be
infinite, and $G$ is always assumed to be finitely generated.

A group action on a tree is {\em minimal} if there exists no proper
invariant subtree. A cocompact action on an infinite tree is minimal if and only if the
tree has no leaves.  
We define a {\em morphism of trees} to mean a map between two trees
which sends vertices to vertices, edges to edges, and preserves
adjacency. All maps between $G$-trees are 
assumed to be $G$-equivariant morphisms of trees, and therefore continuous. 
The {\em star} of a vertex is the set of edges adjacent to
that vertex, and a morphism is {\em locally surjective} (resp.
{\em locally injective}) if for each vertex of the domain tree, the
corresponding map between stars is surjective (resp. injective).  
See \cite{basscoveringtheory} for further discussion. In the context of
morphisms of trees (as opposed to graphs), local injectivity is equivalent to injectivity, 
and local surjectivity implies surjectivity. A tree is {\em locally
finite} if the star of each vertex is finite; such trees are proper
metric spaces.

\begin{theorem}[Main Theorem]
\label{mainthm}
Let $G$ be a finitely generated group, $T$ a locally finite 
tree, and $G \stackrel{\rho}{\curvearrowright} T$ a cocompact action
by isometries. If there exists a minimal $G$-tree $\tilde T$ and a
$G$-morphism $q: \tilde T \rightarrow T$ which is
locally surjective, but not locally injective, then $\Sigma^1(\rho)$
consists of at most a single point of $\boundaryinf T$.
\end{theorem}

We do not require $\tilde T$ to be locally finite, as it
is irrelevant to us whether or not $\tilde T$ is proper as a metric
space. Also, it is worth noting that the map $q: \tilde T \rightarrow
T$ does not generally extend to a map $\boundaryinf \tilde T
\rightarrow   \boundaryinf T$, as geodesic rays may be collapsed to finite paths by $q$.

As mentioned in the introduction, $\Sigma^1(\rho)$ is a
$G$-invariant subset of $\boundaryinf T$. Hence, if the conditions of
the Main Theorem apply and there does exist a point $E_0 \in
\Sigma^1(\rho)$, then $E_0$ is necessarily fixed by $\rho$. In some
cases, this allows us to easily determine that $\Sigma^1(\rho)$ is
empty, as in the following examples. 

\begin{figure}
\includegraphics[scale=.75]{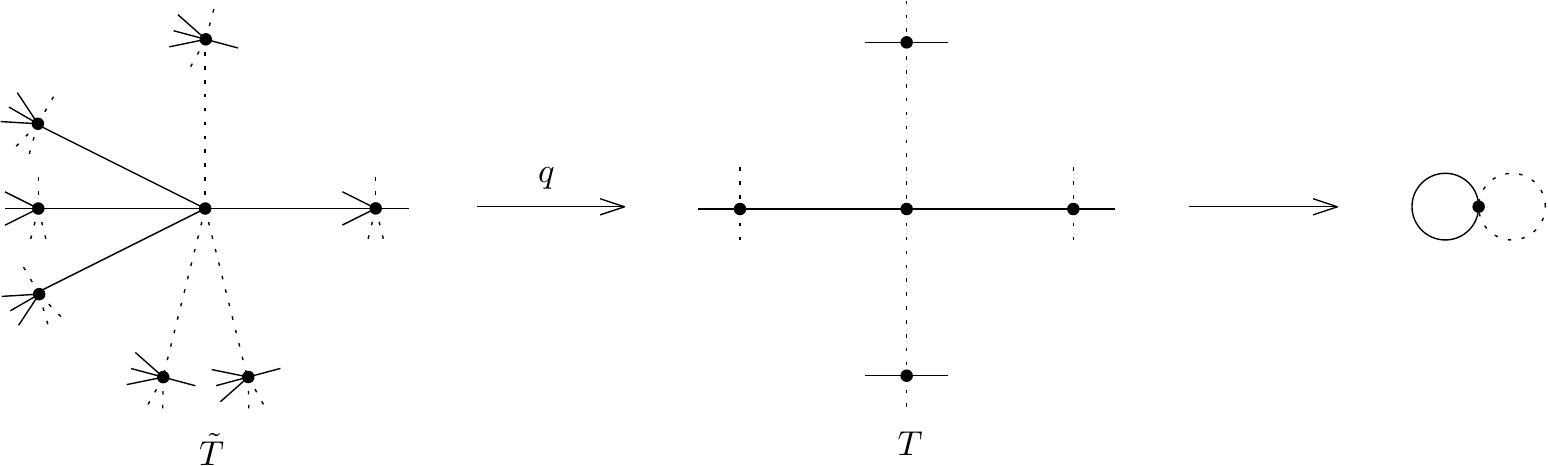}
\caption{$G$ admits a normal subgroup $N$, whose action on $\tilde T$
collapses $\tilde T$ to $T$.}
\label{2roseexfigure}
\end{figure}

\begin{example}
\label{2roseexample}
Let $G$ be the group given by the presentation 
\[ G = \langle a,s,t| a^s = a^2, a^t = a^3 \rangle.\]
As is clear from the given presentation, $G$ can be realized as
a fundamental group of a graph of groups, where the graph is a 2-rose
(a single vertex with two loops). The Bass-Serre tree $\tilde T$ associated 
with this graph of groups decomposition is a regular 7-valent tree.
Let $N$ be the normal closure of $a$. Then $N$ consists of all elements of $G$ which
stabilize a vertex in $\tilde T$. The quotient group $G/N$ is free
on two generators and acts freely on $T = N \backslash \tilde T$ with quotient a
2-rose of circles, whereby $T$ is a regular 4-valent tree. 
Figure \ref{2roseexfigure} demonstrates the collapsing on a neighborhood of a
vertex in $\tilde T$. (One can take $T$ to be the Cayley graph of $G/N$.) The
natural quotient map $\tilde T \rightarrow T$ satisfies 
the conditions of the Main Theorem, and no end point $E \in \boundaryinf T$ is
fixed by $\rho$.  Hence $\Sigma^1(\rho) = \emptyset$.

This example can be generalized to any non-free group with a graph of
groups decomposition over a graph containing a single vertex. Such a
group always has a free quotient obtained by collapsing the normal
closure of the subgroup
associated with the vertex, and as above, the Cayley graph of this
free group can be viewed as the quotient of the original Bass-Serre
tree.
\end{example}

\begin{example}
\label{lehnertexample}
One of Lehnert's counterexamples to the question of whether
$\Sigma^1$ must be either $\emptyset$, a singleton, or $\boundaryinf
T$ in the case of simplicial trees is closely related to
the group $G$ discussed in Example \ref{2roseexample}: Let 
$H = \integers[\frac{1}{6}] \rtimes F_2(x,y)$, where
$F_2(x,y)$ is a free group generated by the letters $x$ and $y$.
One obtains $H$ from $G$ by adding relations corresponding to the commutator subgroup
of $N$.  The semidirect product structure is given by $t^x = \frac{t}{2}$
and $t^y =\frac{t}{3}$ for $t \in \integers[\frac{1}{6}]$. This group acts on
the same tree $T$, by viewing it as the Cayley graph of its factor $F_2(x,y)$, and
one can represent points in $\boundaryinf T$ by infinite reduced words
in $F_2(x,y)$. Any point represented by an infinite word eventually
consisting of only $x$ or only $y$ will not lie in
$\Sigma^1$ \cite{lehnert}. This fact is a consequence of the
interplay between the actions by $F_2(x,y)$ on
$\integers[\frac{1}{6}]$ and on $T$. The author has a 
proof of this result in a paper currently in preparation, which is 
based on the ``topological construction
of the Bass-Serre tree'' \cite{scottandwall} \cite[Ch.6]{geogheganbook} and 
is distinct in flavor from the both the contents
of this paper and the proof in \cite{lehnert}.

Evidently, for the action $H \curvearrowright T$, there exists
no $\tilde T$ and $q: \tilde T \rightarrow
T$ as described in Theorem \ref{mainthm}. 
\end{example}

\begin{example} 
\label{freeproductexample}
Here is an example where $\tilde T$ is not locally finite.
Let $K_4 = \integers_2 \oplus \integers_2$ be the Klein
4-group, and $D_\infty = \integers_2 * \integers_2$ the infinite
dihedral group. Consider the quotient map $\pi:  D_\infty * D_\infty \twoheadrightarrow
K_4 * K_4$, induced by performing the abelianization map $D_\infty
\twoheadrightarrow K_4$ on each free factor of $D_\infty * D_\infty$.
There is an action $\tilde \rho: D_\infty * D_\infty \rightarrow
\autm(\tilde T)$, where $\tilde T$ (a regular $\infty$-valent tree) is the
Bass-Serre tree corresponding to the given free product decomposition.
There is also an action $\rho: D_\infty * D_\infty
\rightarrow \autm(T)$, where $T$, a regular 4-valent tree,
is the Bass-Serre tree for $K_4 * K_4$; this action factors through
$\pi$. We can realize $T$ as a quotient of $\tilde T$ satisfying 
the conditions of the Main Theorem. Again, because no 
point of $\boundaryinf T$ is fixed by $\rho$, it follows that
$\Sigma^1(\rho)$ is empty. This example is of a kind initially
pointed out to the author by Mike Mihalik. 

This, too, can be generalized: if $A_1$ and $A_2$ are two finitely
generated infinite groups, which admit finite quotients $Q_1$ and $Q_2$, respectively, 
then $G = A_1 * A_2$ admits a quotient map $\pi: G \rightarrow Q_1 * Q_2$.
While $G$ acts on the Bass-Serre tree $\tilde T$ corresponding to the
decomposition $A_1 * A_2$, it also acts on $\kerm \pi \backslash
\tilde T$, which is isomorphic to the Bass-Serre tree corresponding to
$Q_1 * Q_2$.
\end{example}

\begin{example}
More generally, there is a notion of a {\em morphism of graphs of
groups} (essentially, a morphism of graphs together with a collection
of homomorphisms of vertex and edge groups that ensure certain squares
commute), which lifts to an equivariant morphism between the
corresponding Bass-Serre trees (Proposition 2.4 of
\cite{basscoveringtheory}), and one can determine whether the lift
will be locally surjective and not locally injective (Corollary 2.5 of
\cite{basscoveringtheory}). 
This can be used to produce maps satisfying the
conditions of Theorem \ref{mainthm}. For example, consider the
Baumslag-Solitar groups $BS(m,n) = \langle a, t \mid ta^m\inv{t} = a^n\rangle$. 
There is a projection map $BS(2,4) \twoheadrightarrow BS(1,2)$
obtained by adding the relation $ta\inv ta^{-2}$. One can show that
this corresponds to a morphism of graphs of groups which lifts to a
map between the corresponding Bass-Serre trees and has the desired
properties.
\end{example}

Applying Theorems $A$ and $H$ of \cite{amsmemoir}, we have:

\begin{corollary}
If $G \stackrel{\rho}{\curvearrowright} T$ satisfies the conditions of
the Main Theorem, then for any point $z \in T$, the stabilizer $G_z$ of $z$ under the action $\rho$
is not finitely generated. {\raggedright\qed}
\end{corollary}

\subsection{Collapsing Pairs} Recall that, in the language of
\cite{serretrees}, Chapter I.2, each geometric edge of $T$ corresponds
to two oriented edges, one pointing in either direction.
\begin{remark} We will use the lowercase $e$ to refer to edges of $T$,
oriented or not,  and the 
uppercase $E$ to refer to points of $\boundaryinf T$.
\end{remark}

\begin{definition}
Under the hypotheses of the Main Theorem, let
$(\tilde e_1, \tilde e_2)$ be a pair of
adjacent distinct oriented edges in $\tilde T$ with common initial vertex $\tilde v$.
If $q(\tilde e_1) = q(\tilde e_2)$,
we call this a {\em collapsing pair (of edges)}
under $q$.  Let $e = q(\tilde e_1)$  be the
resulting oriented edge in $T$.
For a vertex $w \in T$ (or end point $E \in \boundaryinf T$), we say the pair
$(\tilde e_1, \tilde e_2)$ {\em faces} $w$ (resp., $E$) if $e$ points
toward $w$ (resp., $E$). This is the same as saying the geodesic from
$q(\tilde v)$ to $w$ (resp., $E$) passes through $e$.
\end{definition}


The proof of the Main Theorem will follow from two facts:
Propositon \ref{onlyoneendlemma} states that 
because $q$ is not locally injective, all end points of $T$ (with the possible exception of a
single end point) are faced by a collapsing pair. Proposition
\ref{pairfacesendlemma} states that local
surjectivity of $q$ forces any end point of $T$ faced by a collapsing
pair to lie outside $\Sigma^1(\rho)$. 

\subsection{The case where stabilizers on $\tilde T$ have type $F_n$}
If we add the condition that the stabilizers under $\tilde \rho$ have
type $F_n$, then we can prove that 
a point $E \in \boundaryinf T$ which is {\em not} faced by a
collapsing pair lies in $\Sigma^n(\rho)$.

\begin{theorem}
\label{theorem2}
Assume the conditions of the Main
Theorem. Furthermore, suppose that for $n > 0$,
$G$ has type $F_n$ and for each point $\tilde z$ of $\tilde T$, the stabilizer 
$G_{\tilde z}$ has type $F_n$.
Then $E \in \boundaryinf T$ lies in $\Sigma^n(\rho)$ if and only if
there is no collapsing pair facing $E$.
\end{theorem}

\begin{corollary} Let the group $H$ have type $F_n$, and let 
$\varphi: H \rightarrow H$ be injective, so that 
$G = \langle H, t \mid a^t = \varphi(a)\ \forall\ a \in H \rangle$ is an ascending 
HNN-extension.  If $\chi: G \twoheadrightarrow \integers$ maps $t \mapsto 1$
and $\langle \langle H \rangle \rangle \mapsto 0$, then $\chi$ represents
a point in $\Sigma^n(G)$. {\raggedright \qed}
\end{corollary}

This corollary is not new \cite{BieriStrebel} \cite{meinert96}
\cite{meinert97}, but the approach is. For further
discussion on this result, see \cite{bieri-2008}.

\section{Controlled connectivity}
\label{controlledconnectivitysection}
In a CAT(0) space $M$ there is a notion of a 
{\em (visual) boundary} $\boundaryinf M$ obtained by taking equivalence classes
 of geodesic rays \cite[Ch. II.8]{bridsonhaefliger}. This
boundary carries a topology, called the cone topology, induced by the
topology on $M$.  We call points of $\boundaryinf M$ {\em end points}.
CAT(0) spaces are contractible, and the boundary of a proper CAT(0)
space is a compact space. 
Let $\tau$ be a geodesic ray in $M$.
Following \cite{amsmemoir}, we
define the {\em Busemann function} $\beta_\tau: M \rightarrow
\reals$ by
\[\beta_\tau(p) = \lim_{t\rightarrow \infty}( t -
d(\tau(t), p)).\]
For $r \in \reals$, the set 
$HB_r(\tau) = \inv{\beta_\tau}([r,\infty))$
is called a {\em horoball} around $E$. Horoballs in CAT(0) spaces are
contractible. We can view $HB_r(\tau)$ as the nested union of closed
balls $\cup_{k\geq \max\{0,r\}} \overline{B_{k-r}(\tau(k))}$.


\begin{definition}
Fix $n \in \naturals$. Let $G$ be a group having type $F_n$
and let $M$ be a proper CAT(0) space
admitting an isometric action $G \stackrel{\rho}{\curvearrowright} M$.
Choose an $n$-dimensional $(n-1)$-connected CW-complex $X^n$ on which
$X$ acts freely and cocompactly, and choose a continuous $G$-map 
$h : X^n \rightarrow M$. We call $h$ a {\em control map};
one can be found because the action by $G$ on $X^n$ is free and $M$
is contractible. Fix a geodesic ray $\tau$ representing $E \in
\boundaryinf M$. For a
horoball $HB_r(\tau)$ about $E$,  denote the largest subcomplex of
$X^n$ contained in $\inv{h}(HB_r(\tau))$ by $X_{(\tau,r)}$.
Finally, we need a notion of {\em lag function}: any $\lambda(r) > 0$ satisfying 
$r - \lambda(r) \rightarrow \infty$ as $r \rightarrow \infty$ is
called a lag.

We say $\rho$ is
 {\em controlled $(n-1)$-connected, or $CC^{n-1}$, over $E$} if for
all $r \in \reals$  and all $-1 \leq p \leq (n-1)$,
there exists a lag $\lambda$ such that every 
map $f: S^p \rightarrow
X_{(\tau, r)}$ extends to a map $\tilde f: B^{p+1} \rightarrow X_{(\tau,
r-\lambda(r))}$.\footnote{By convention $S^{-1} = \emptyset$, 
and $(-1)$-connected means ``non-empty''.}
\end{definition}

\begin{definition}
The Bieri-Geoghegan invariant $\Sigma^n(\rho)$ is the subset of
$\boundaryinf M$ consisting of all end points over which $\rho$ is
controlled $(n-1)$-connected.
\end{definition}

\subsection{Relationship to the BNSR invariant}


If $\rho$ fixes
an endpoint $E$, then the pair $(\rho,E)$ determines a homomorphism 
$\chi_{\rho,E}: G \rightarrow \reals$, and $E \in \Sigma^1(\rho)$
iff $\chi_{\rho,E}$ represents a point in $\Sigma^1(G)$ \cite[\S 10.6]{amsmemoir}.
In fact, we can obtain the classical BNSR invariant $\Sigma^n(G)$ as the
special case where $\rho$ is the action $G \curvearrowright G_{ab}
\otimes \reals$ \cite[Ch.~10, Example A]{amsmemoir}. This is an
action by translations on a finite dimensional real vector space, so
every end point is fixed, and $\boundaryinf(G_{ab} \otimes
\reals) \iso \textrm{Hom}(G,\reals)$.

The question of finding a single technique for calculating $\Sigma^1$
for arbitrary group actions on trees seems out of reach at this time. To see
this, consider an action $G \stackrel{\rho}{\curvearrowright} T$ by
translations, where $T$ is a simplicial line.  This corresponds to
a homomorphism $\chi: G \twoheadrightarrow \integers$, and
calculating $\Sigma^1(\rho)$ determines whether $\chi$ and $-\chi$
represent points of $\Sigma^1(G)$. 
However, it is known that $\kerm \chi$ is finitely generated if and only
if both do represent points of $\Sigma^1(G)$ \cite[Theorem B1]{BieriNeumannStrebel}.
Thus a method for calculating $\Sigma^1(\rho)$ even in the special case that
the tree is a simplicial line would enable us to determine whether or not
the kernel of an arbitrary homomorphism to $\integers$ is finitely generated.

\section{Proof the Main Theorem}
An automorphism $s$ of a tree $T$ having no fixed point is said to be
{\em hyperbolic}. For each such $s$, there
is a unique line $A_s$, called the {\em axis} of $s$, stable under 
the action of the subgroup $\langle s \rangle$, which acts on $A_s$ by 
translations.  If $e$ is an oriented edge of $T$, then $s$ is said to {\em act
coherently} on $e$ if $e$ and $se$ are consistently oriented (i.e., if
they point in the same direction  --- neither toward each other nor away from each other). 
For an automorphism $s$, if $e \neq se$, then $s$ acts coherently on
$e$ if and only if $s$ is hyperbolic and both $e$ and $se$
lie on the axis of $s$
\cite[Proposition 25]{serretrees}.

\begin{lemma}
\label{onepairmanypairslemma}
Let $T$ be a cocompact $G$-tree, and let $E \in \boundaryinf T$.
Then for any geodesic ray $\tau$ representing $E$, any $r \in \reals$,
and any oriented edge $e$ of $T$ oriented toward $E$, there exists 
an element of the $G$-orbit of $e$ which is oriented toward $E$ 
and does not lie in $HB_r(\tau)$.
\end{lemma}
\begin{proof}
The ray of oriented edges
beginning at $e$ and representing $E$, with all edges pointing toward
$E$, contains infinitely many edges. Because the action is
cocompact, the pigeon-hole principle ensures that
there must be edges $e_1$ and $e_2$ from this ray in the same $G$-orbit.
Hence, there is an $h \in G$ with $he_1 = e_2$. Because $e_1$ and $e_2$ are
consistently oriented, $h$ is hyperbolic. Let $v_1$ be the terminus of
$e_1$ (the vertex of $e_1$ where $\beta_\tau$ is maximized).  By choosing $k
\in \integers$ such that $(i)$ $|k| > \beta_\tau(v_1) - r$ and 
$(ii)$ $h^k$ moves $e_1$ away from $E$, we ensure that $h^ke_1$ is
oriented toward $E$ and does not lie in $HB_r(\tau)$. 
Thus $h^ke$ is the edge we seek.
\end{proof}

\begin{observation}
\label{liftingrayslemma}
For trees $\tilde T$ and $T$, let $q: \tilde T \rightarrow T$ be locally
surjective. If $\tau = (e_0, e_1, \dots)$ is a geodesic edge ray in
$T$ and $\tilde e_0$ is an edge of $\tilde T$ satisfying 
$q(\tilde e_0) = e_0$, then there exists a
lift $\tilde \tau$ of $\tau$ to $\tilde T$ having initial edge
$\tilde e_0$ and which is also a geodesic edge ray.
{\raggedright \qed}
\end{observation}

\begin{observation}
\label{gmapontoobs}
Given a nonempty connected $G$-graph $\Gamma$ and a minimal $G$-tree $T$, any
$G$-morphism $h: \Gamma \rightarrow T$ is surjective.
\end{observation}

\begin{proposition}
\label{pairfacesendlemma}
Let $T$ be a cocompact $G$-tree and let $\tilde T$ be a minimal
$G$-tree. Suppose $q: \tilde T \rightarrow T$ is a $G$-morphism
which is locally surjective.
If $E \in \boundaryinf T$ is such that there exists a collapsing pair
facing $E$, then $E$ does not lie in $\Sigma^1(\rho)$.
\end{proposition}

\begin{proof}
Let $\Gamma$ be a free cocompact $G$-graph, and
choose any $G$-morphism $h: \Gamma \rightarrow \tilde T$. Then the
composition $q \circ h$ is a suitable control map for determining
controlled connectivity over $E$.

Let $\tau: [0, \infty) \rightarrow T$ be a geodesic edge ray
representing $E$. We will show that for any lag $\lambda > 0$, there
exist points in the 
subgraph $\Gamma_{(\tau,0)}$ that cannot be connected via a path in
$\Gamma_{(\tau,-\lambda)}$.

By Lemma \ref{onepairmanypairslemma}, we can choose a collapsing pair
 $(\tilde e_1, \tilde e_2)$ facing $E$ but
whose image in $T$ does not lie in $HB_{-\lambda}(\tau)$. 
Let $\tilde v$ be the vertex shared by
$\tilde e_1$ and $\tilde e_2$, and let $v$ be its image in $T$. 
Let $\gamma$ be the geodesic ray
representing $E$ and emanating from $v$. By Observation \ref{liftingrayslemma}
there exist two distinct lifts $\tilde \gamma_i$,
$i=1,2$, of $\gamma$ to $\tilde T$, with $\tilde \gamma_i$ having initial
edge $\tilde e_i$. Because $\gamma$ and $\tau$
both represent $E$, they eventually merge, so that $\gamma$
intersects $HB_r(\tau)$ nontrivially for all $r \in \reals$.
 Hence, both $\tilde \gamma_1$
and $\tilde \gamma_2$ intersect
$\inv{q}(HB_r(\tau))$ for all $r$. 

By design, $\tilde \gamma_1 \cap \tilde \gamma_2 = \tilde v$, and
$\tilde \gamma_1 \cup \tilde \gamma_2$ is a line. By Observation
\ref{gmapontoobs}, $h$ is onto,
so that $\tilde \gamma_1 \cup \tilde \gamma_2$ lies in
the image of $h$. For $i = 1,2$, choose a vertex $\tilde y_i \in \tilde \gamma_i
\cap \inv{q}(HB_0(\tau))$, and choose $x_i \in
\inv{h}(\tilde y_i)$. Then both $x_i$ lie in $\Gamma_{(\tau, 0)}$, but any
path through $\Gamma_{(\tau,-\lambda)}$ joining $x_1$ to $x_2$ would be
mapped to a path in $\inv{q}(HB_{-\lambda}(\tau))$ joining $\tilde y_1$ to
$\tilde y_2$. Since $\tilde T$ is a tree, no such path exists. 
\end{proof}

\begin{lemma}
\label{invariantedgesetlemma}
Let $T$ be a minimal $G$-tree and let $\mathcal E$ be a nonempty
$G$-invariant set of oriented edges. Then there is no vertex $v$ in $T$
such that all edges of $\mathcal E$ are oriented away from $v$.
\end{lemma}
\begin{proof}
The full subtree of $T$ on the vertex subset
\[ \{ v \mid \text{each edge of } \mathcal E \text{ is oriented away
from from }v\} \]
is a proper $G$-invariant subtree. By minimality, this set must be
empty.
\end{proof}

\begin{corollary}
\label{verticesfacedcor}
Let $T$ be a cocompact $G$-tree and let $\tilde T$ be a minimal
$G$-tree. Suppose $q: \tilde T \rightarrow T$ is a $G$-morphism which is surjective but not
locally injective. Then every vertex of $T$ is faced by a collapsing
pair.
\end{corollary}
\begin{proof}
Let $\tilde{\mathcal E}$ be the set of oriented edges of $\tilde T$
that are part of a collapsing pair. This is a $G$-invariant set, and it is
nonempty because $q$ is not locally injective. By
Lemma \ref{invariantedgesetlemma} each vertex $\tilde v$ of $\tilde T$
must therefore have an edge $\tilde e$ in $\tilde{\mathcal E}$
oriented toward $\tilde v$. Set $v = q(\tilde v)$. Then
if  $q(\tilde e)$ is not oriented toward $v$, the image of
the path from $\tilde e$ to $\tilde v$ must contain points of
backtracking. The point of backtracking closest to $v$ gives
rise to a collapsing pair facing $v$.
Because $q$ is surjective, all vertices of $T$ are of this form.
\end{proof}

\begin{observation} 
\label{sameendsobs}
If a cocompact $G$-tree $T$ has a nonempty $G$-invariant
subtree $T'$, then $T$ is a Hausdorff neighborhood of $T'$. Hence, $T$
and $T'$ have the same set of end points.
\end{observation}

\begin{proposition}
\label{onlyoneendlemma}
Let $T$ be a cocompact $G$-tree and let $\tilde T$ be a minimal $G$-tree. 
Suppose $q: \tilde T \rightarrow T$ is a $G$-morphism which is not locally injective. 
Then there exists at most one point $E_0 \in \boundaryinf T$ such that no 
collapsing pairs face $E_0$.   
\end{proposition}

\begin{proof}
By Observation \ref{sameendsobs}, the ends of $T$ and the ends of
$q(\tilde T)$ are the same, so we may assume $q$ is surjective. 
By Corollary \ref{verticesfacedcor}, each vertex of $T$ is faced by a
collapsing pair in $\tilde T$. If two points of $\boundaryinf T$ were
not faced by a collapsing pair, then no vertex on the line between
them would be faced by a collapsing pair. Hence, there can be at most
one point of $\boundaryinf T$ not faced by a collapsing pair. 
\end{proof}

This proposition has an interesting consequence. If such an end $E_0$
exists, it must clearly be fixed by $\rho$. Yet points of the
boundary which are fixed by $\rho$ correspond to homomorphisms $G
\rightarrow \reals$, and such an end point lies in $\Sigma^n(\rho)$ if
and only if the corresponding homomorphism lies in the BNSR invariant
$\Sigma^n(G)$, as discussed in \S2.  Since we only consider simplicial trees, such 
points in fact correspond to homomorphisms $G \twoheadrightarrow \integers$. This leads
to the following corollary:

\begin{corollary}
\label{correspondingcharactercor}
Under the conditions of Proposition \ref{onlyoneendlemma},
if an end point $E_0 \in \boundaryinf T$ is faced by no collapsing
pair in $\tilde T$, then there exists a canonically associated discrete character
$\chi: G \rightarrow \integers$ such that $E_0 \in \Sigma^n(\rho)$ if
and only if $[\chi] \in \Sigma^n(G)$, the BNSR invariant. {\raggedright\qed}
\end{corollary}

\begin{proof}[Proof of Main Theorem]
Because $q$ is not locally injective, Proposition
\ref{onlyoneendlemma} ensures there is at most one end point
faced by a collapsing pair. Because $q$ is locally surjective, 
Proposition \ref{pairfacesendlemma} ensures that every end point faced
by a collapsing pair lies outside $\Sigma^1(\rho)$. 
\end{proof}

\subsection{The case where stabilizers under $\tilde \rho$ have type
$F_n$}
Recall the ``topological construction
of the Bass-Serre tree'', discussed in \S6.2 of \cite{geogheganbook}, 
and in \cite{scottandwall}: the action 
$\tilde \rho$ corresponds to a graph of groups decomposition of $G$.
From this we can build a $K(G,1)$ $X$ admitting the quotient $G\backslash
\tilde T$ as a retract.  Let $p: \tilde X \twoheadrightarrow X$ be the
universal covering projection. 
There is a natural $G$-map $h: \tilde X
{\twoheadrightarrow} \tilde T$, and it is clear from the
construction of $h$ that for any
connected subset $A \subseteq \tilde T$, $\inv{h}(A) \subseteq \tilde X$ is
contractible. If for an integer $n \geq 1$ all point stabilizers under
$\tilde \rho$ have type $F_n$, then we can
take $X$ to have compact $n$-skeleton.
Hence, letting $\Gamma$ be the $n$-skeleton of $\tilde X$,
the composition $\bar h = q \circ h|_\Gamma: \Gamma \rightarrow T$ is an appropriate control map for
$\rho$.

\begin{definition}While the map $q$ does not induce a map $\boundaryinf \tilde T
\rightarrow \boundaryinf T$, each geodesic ray in $T$ can be lifted
to one or more geodesic rays in $\tilde T$ (see Observation
\ref{liftingrayslemma}) as long as $q$ is locally surjective. 
Hence, given $E \in
\boundaryinf T$, we can consider the set $\inv{q}(E) \subseteq
\boundaryinf \tilde T$ of end points represented by lifts of rays
representing $E$. 
\end{definition}

\begin{lemma} If $q$ is locally surjective, then $\inv{q}(E)$ is a singleton if and only
if there are no collapsing pairs facing $E$.
\end{lemma}

\begin{proof}
Suppose that $\inv{q}(E)$ is not a singleton. Then for $\tau$
representing $E$, there exist two distinct lifts $\tilde \tau_1$ and
$\tilde \tau_2$, representing distinct points $\tilde E_1$ and
$\tilde E_2$ of $\boundaryinf \tilde T$. If these lifts are not
disjoint, then where they split (as they must, eventually) there is a
collapsing pair facing $E$. If they are disjoint, consider the geodesic
path $P$ through $\tilde T$ connecting them. The image of $P$ in $T$ is a
finite subtree of $T$. Choose any vertex $v \neq \tau(0)$ which is
a leaf of this subtree. This leaf and the corresponding edge lie
under a collapsing pair of edges of $P$ facing $E$. 

Now suppose there is a collapsing pair $(\tilde e_1, \tilde e_2)$ of edges of
$\tilde T$ facing $E$.  Let $e$ be their common image in $T$, and let
$\zeta$ be the geodesic ray in $T$ representing $E$ and beginning with
the edge $e$. Then there are distinct lifts $\tilde \zeta_1$ and
$\tilde \zeta_2$ of $\zeta$, each representing a distinct end point of
$\tilde T$. Hence $\inv{q}(E)$ is not a singleton.
\end{proof}

\begin{proof}[Proof of Theorem \ref{theorem2}]
If there is a collapsing pair facing $E$, then by 
Proposition \ref{pairfacesendlemma}, $E \nin \Sigma^1(\rho)$. 

If there is no collapsing pair facing $E$, we take the control map
$\bar h$ described above.
By construction of $\bar h$, we need only show that for any horoball
$HB_r(\tau)$ about $E$,  $\inv{q}(HB_r(\tau))$ is connected. 

For $i = 1,2$, let $\tilde z_i$ be a point in $\inv{q}(HB_r(\tau))$,
and let $z_i$ be its image in
$T$. We will find a path  between $\tilde z_1$ and $\tilde z_2$ lying in
$\inv{q}(HB_r(\tau))$. 

For $i=1,2$, there exists a unique geodesic ray $\zeta_i$ in $T$ which
emanates from $z_i$ and represents $E$. Let $\tilde \zeta_i$ be the lift of
$\zeta_i$ to $\tilde T$ emanating from $\tilde z_i$. Since $\zeta_i$
lies in $HB_r(\tau)$, $\tilde \zeta_i$ lies in $\inv{q}(HB_r(\tau))$.
Moreover, since $\inv{q}(E)$ is a singleton, $\tilde \zeta_1(\infty) = \tilde
\zeta_2(\infty)$. Hence, $\tilde \zeta_1$ and $\tilde \zeta_2$ must
eventually merge. The closure of 
$(\imm \tilde \zeta_1 \cup \imm \tilde \zeta_2) - (\imm \tilde
\zeta_1 \cap \imm \tilde \zeta_2)$
is the geodesic connecting $\tilde z_1$ to $\tilde z_2$.
\end{proof}

\bibliographystyle{plain}
\bibliography{paper7}

\end{document}